\documentclass[10pt, a4paper]{amsart}
\usepackage[english]{babel}
\usepackage{amssymb,amsmath,amsthm}
\usepackage{verbatim}

\numberwithin{equation}{section}

\newtheorem{dummy}{dummy}[section]

\newtheorem{theorem}[dummy]{Theorem}
\newtheorem{corollary}[dummy]{Corollary}

\newtheorem{lemma}[dummy]{Lemma}

\newtheorem{definition}[dummy]{Definition}

\def\Co{\mathbb C}

\def\Q{\mathbb Q}

\def\Z{\mathbb Z}
\def\e{\varepsilon}

\def\l{\lambda}
\def\a{\alpha}
\def\b{\beta}
\def\={\;=\;}
\def\+{\,+\,}
\def\-{\,-\,}
\def\bal{\begin{aligned}}
\def\eal{\end{aligned}}
\def\be{\begin{equation}\label}
\def\ee{\end{equation}}

\def\L{{\mathcal L}}

\def\ord{{\rm ord}}
\title{Equations D3 and spectral elliptic curves}
\author{Vasily Golyshev}
\author{Masha Vlasenko}
\address{Algebra and Number Theory Sector, Institute for Information Transmission Problems, B. Karetny 19, Moscow 127994, Russia}
\email{golyshev@mccme.ru}
\address{School of Mathematics, Trinity College, Dublin 2, Ireland}
\email{masha.vlasenko@gmail.com}

\begin{document}
\begin{abstract}
We study modular determinantal differential equations of orders 2 and 3. We show that   the expansion of the analytic solution of a non—degenerate modular equation of type D3 over the rational numbers with respect to the natural parameter  coincides, under certain assumptions,  with the $q$--expansion of the newform of its spectral elliptic curve and therefore possesses a multiplicativity property. We compute the complete list of D3 equations with this multiplicativity property and relate it to Zagier's list of non—degenerate modular D2 equations. 

\end{abstract}

\maketitle

\section{Introduction}
Motivated by Ap\'{e}ry's proof of irrationality of $\zeta(3)$, Don Zagier studies in~\cite{Zag} the question of finding those triples of rational numbers $(A,B,\l)$ for which the sequence obtained by the recursive formula
\[
(n+1)^2 u_{n+1} \- (An^2 + A n + \l)u_n \+ B n^2 u_{n-1} \= 0  
\] 
starting with $u_0=1$ has all integer terms, i.e. $u_n \in \Z$. The generating function $\phi_0(t)=1+u_1 t + u_2 t^2 + \dots$ is the normalized  analytic at $t=0$ solution of the differential equation $\L \phi_0 \= 0$ with
\be{D2zag}
\L(t) \= D^2 \- t \, (A D^2 + A D + \l) \+ B \, t^2 \, (D+1)^2   \,,
\ee
where we use the notation $D \= t \frac {d}{dt}$ throughout the paper. We will refer later on to~\eqref{D2zag} as the Beukers-Zagier differential operator since it appeared in the works of these two authors. A table of respective  triples  $(A,B,\l)$   is obtained in~\cite{Zag} by searching in a large range of values. It appears that all degenerate  cases in the table, i.e. those 
with either $A^2 = 4 B$ or $ B = 0$
come as members of infinite families of triples $(A,B,\l)$ with $\phi_0(t) \in \Z[[t]]$. 
On the contrary, imposing the assumption 
\be{nondeg}
A^2 \ne  4 B, \;  B \ne 0
\ee
one arrives at 14 ``sporadic'' cases with no obvious pattern. Remarkably, in all those sporadic cases the corresponding differential equation can be parametrized by modular forms. Namely, one can find a modular function $t(\tau)$ that vanishes at $\infty$ and a modular form $f(\tau)$ of weight~1  such that $\phi_0(t(\tau))=f(\tau)$ for all $\tau$ in the upper half-plane with large enough imaginary part. These cases are listed in the table below. 
 
\begin{center}
\begin{tabular}{ccc|cc|cccccc}
&&&&&&&&&&\\
$A$&$B$& $\l$&$t(\tau)$ & $f(\tau)$ &$u_0$& $u_1$&$u_2$& $u_3$&$u_4$ & $u_5$\\
\hline
&&&&&&&&&&\\
7&-8&2&  ${\bf \frac{1^36^9}{2^33^9}}$&${\bf \frac{2^13^6}{1^26^3}}$& 1 &2 &10 &56&346&2252\\
&&&&&&&&&&\\
9&27&3&  ${\bf \frac{1^34^318^9}{2^99^336^3}}$&${\bf \frac{2^93^112^1}{1^34^36^3}}$& 1 &3 &9 &21&9&-297\\
&&&&&&&&&&\\
10&9&3&  ${\bf \frac{1^46^8}{2^83^4}}$&${\bf \frac{2^63^1}{1^36^2}}$& 1 &3 & 15&93&639&4653\\
&&&&&&&&&&\\
11&-1&3& $q \prod\limits_{n=1}^{\infty}(1-q^n)^{5 (\frac{n}{5})}$ &$\Bigl(\frac1{t(\tau)}\cdot {\bf \frac{5^5}{1^1}}\Bigr)^{1/2}$&1  &3 &19 &147&1251&11253\\
&&&&&&&&&&\\
12&32&4&  ${\bf \frac{1^44^28^4}{2^{10}}}$&${\bf \frac{2^{10}}{1^44^4}}$& 1 &4 &20 &112&676&4304\\
&&&&&&&&&&\\
17&72&6&  ${\bf \frac{1^53^14^56^212^1}{2^{14}}}$&${\bf \frac{2^{15}3^212^2}{1^64^66^5}}$& 1 &6 &42 &312&2394&18756\\
&&&&&&&&&&\\
0&-16&0&  ${\bf \frac{2^48^8}{4^{12}}}$&${\bf \frac{4^{10}}{2^48^4}}$&  1&0 &4 &0&36&0\\
\end{tabular}
\end{center}

\noindent
The products in the forth and fifth columns stand for eta-products, e.g. ${\bf \frac{1^36^9}{2^33^9}} = \frac{\eta(\tau)^3\eta(6\tau)^9}{\eta(2\tau)^3\eta(3\tau)^9}$. We use this notation throughout the paper. For each of the first 6 rows one should also consider $(\tilde{A},\tilde{B},\tilde{\l})=(-A,B,-\l)$ with the corresponding $\tilde{u}_n=(-1)^n u_n$, $\tilde{t}(\tau)=-t(\tau)$, $\tilde{f}(\tau)=f(\tau)$. For the last row there is also a triple $(0,16,0)$ leading to $\tilde{u}_n=(-1)^{n/2} u_n$. In total, the table gives us  14 triples $(A,B,\l)$ satisfying~\eqref{nondeg} with $\phi_0(t)\in\Z[[t]]$. Zagier conjectures that there are no more such cases, or if there are, they all will have modular parametrization.

We observe in this paper that the differential operator~\eqref{D2zag} satisfying the assumption~\eqref{nondeg} is a specific form of the so called determinantal differential operator of order~2. The necessary definitions and properties will be recalled in Section~\ref{sec:DN}. We deal with determinantal differential operators of order~2 in Sections~\ref{sec:D2} and~\ref{sec:D2modular} and recover Zagier's list above in a new context. Then we proceed  to obtain an analog of this list for determinantal differential operators of order~3, namely, the complete list of D3's that satisfy a multiplicativity property that we discuss later.

\section{Determinantal differential equations}\label{sec:DN}

Determinantal differential equations of order $N$ were defined in~\cite{GS}. 
A DN equation is obtained from  an $(N+1)\times(N+1)$ matrix $A=(a_{ij})_{i,j=0}^N$ 
that satisfies
\be{asA}\bal
&a_{ij}\=0\,, \quad i-j>1 \\
&a_{ij}\=1\,, \quad i-j=1 \\
&a_{ij}\=a_{N-j,N-i}\,, \quad i-j<1 \\
\eal\ee
The respective differential operator is then defined as 
\[
\L_{A,\infty}(z) \= \det\,_{\mathrm{right}} \,\Bigl( \delta_{ij} z \frac{d}{dz} - a_{ij} \bigl(\frac{d}{dz}\bigr)^{j-i+1} \Bigr) \, \bigl(\frac{d}{dz}\bigr)^{-1}
\]
where $\delta_{ij}$ is the Kronecker symbol and $\det_{\mathrm{right}}$ refers to the way of expanding the determinant of a matrix with non-commuting entries with respect to the rightmost column.

The matrix $A$ can be reconstructed from the coefficients of the differential operator (\cite{GS}, Corollary 3.3). Assume in addition that all eigenvalues of $A$ are distinct. Then obviously $A$ is diagonalizable. In fact, for a matrix satisfying~\eqref{asA} the two conditions are equivalent: $A$ is diagonalizable if and only if all eigenvalues of $A$ are distinct. It follows immediately if one observes that $A$ cannot have an eigenvector whose  last component iz zero. 
According to Corollary 6.4 in~\cite{GS} the singularities of the differential operator $\L_{A,\infty}(z)$ are regular singular points located at $\infty$ and the eigenvalues $\lambda_0,\dots,\lambda_N$ of $A$. Moreover, the differential equation has maximal unipotent monodromy at $z=\infty$ and the valuation of its analytic solution at $z=\infty$ is equal to $1$. This motivates the following notation.

\begin{definition} A differential operator of order $N$ is of type $DN_{\infty,1}$ if it equals $\L_{A,\infty}(z)$ for some matrix $A$ satisfying~\eqref{asA}.\end{definition}

\bigskip

We denote the characteristic polynomial of $A$ by $F(z)=\det(z-A)$ 
throughout the paper.  It will be convenient to also use the variable $t = \frac1z$. Namely, consider the operator 
\[
\L_{A,0}(t) \= (-1)^N \, \L_{A,\infty}\Bigl( \frac1t \Bigr) t\,.
\]  
The respective differential equation has maximal unipotent monodromy at $t=0$ and the valuation of its analytic solution at this point equals $0$. 

\begin{definition} A differential operator of order $N$ is of type $DN_{0,0}$ if it equals $\L_{A,0}(t)$ for some matrix $A$ satisfying~\eqref{asA}.
\end{definition}

\bigskip

By DN we mean either $DN_{0,0}$ or $DN_{\infty,1}$ the case being clear from the context. Observe that the following operations  with the defining  matrices
\[\bal
& A \mapsto A' = A + \e \\
& A \mapsto A'' = \Bigl( \l^{j-i+1} a_{ij} \Bigr)
\eal\]
lead to the substitutions in the differential equations 
\be{shiftDN}
\L_{A',\infty}(z) \= \L_{A,\infty}(z - \e)\,, \qquad \L_{A',0}(t) \= \L_{A,0}\Bigl( \frac{t}{1-\e t}\Bigr) \, (1-\e t)
\ee
and 
\be{scaleDN}
\L_{A'',\infty}(z) \= \l \, \L_{A,\infty}\Bigl(\frac{z}{\l}\Bigr)\,,\qquad \L_{A'',0}(t) \= \L_{A,0}( \l t)
\ee
respectively.

\section{The Beukers-Zagier equation as a D2 equation}\label{sec:D2}

Let us consider D2 equations in detail. According to our definitions one has
\[\bal
&\L_{A,\infty}(z) \= \det\,_{right} \, \begin{pmatrix} 
(z-a_{00})\frac{d}{dz} & -a_{01}\bigl(\frac{d}{dz}\bigr)^2 & -a_{02}\bigl(\frac{d}{dz}\bigr)^3 \\
-1 & (z-a_{11})\frac{d}{dz} & -a_{01}\bigl(\frac{d}{dz}\bigr)^2 \\
0 & -1 & (z-a_{00})\frac{d}{dz}\\
\end{pmatrix} \; \bigl(\frac{d}{dz}\bigr)^{-1}\\
&\= -a_{02}\bigl(\frac{d}{dz}\bigr)^2 -a_{01}\bigl(\frac{d}{dz}\bigr)^2 (z-a_{00}) + (z-a_{00})\frac{d}{dz}\Bigl( (z-a_{11})\frac{d}{dz}(z-a_{00}) -a_{01}\frac{d}{dz}\Bigr)\\
&\= F(z) \, \bigl(\frac{d}{dz}\bigr)^2 \+ F'(z) \, \frac{d}{dz} \+ (z-a_{00})
\eal\]
where
\[\bal
F(z) &\= \det\bigl(z-A\bigr) \= z^3 \+ \a_2 z^2 \+ \a_1 z \+ \a_0 \\
&\a_2 \= -a_{11}-2a_{00}\\ 
&\a_1 \= 2 a_{00} a_{11} + a_{00}^2 - 2 a_{01}\\
&\a_0 \= 2 a_{00} a_{01} - a_{00}^2 a_{11} - a_{02}\\
\eal\]
A $D2_{\infty,1}$ differential operator is then any operator of the form
\[
F(z) \, \bigl(\frac{d}{dz}\bigr)^2 \+ F'(z) \, \frac{d}{dz} \+ (z- \b) 
\]
with a cubic monic polynomial with distinct roots $F(z) \= z^3 \+ \a_2 z^2 \+ \a_1 z \+ \a_0$. One can recover the matrix $A$ from $\a_i$ and $\b$ via
\[\bal
&a_{00} \= \b \\
&a_{11} \=  - 2 \b - \a_2 \\
&a_{01} \= -\frac32 \b^2  - \b \a_2 -\frac12 \a_1 \\
&a_{02} \= -\b^3-\b^2\a_2 -\b a_1 -\a_0
\eal\]

\bigskip

The generic equation of type $D2_{0,0}$ is then
\[\bal
&F\Bigl(\frac1t\Bigr) \Bigl(-t^2 \frac{d}{dt}\Bigr)^2 t \+ F'\Bigl(\frac1t\Bigr) \Bigl(-t^2 \frac{d}{dt}\Bigr)t \+ \Bigl(\frac1t - \b\Bigr) t \\
& \= t^5 F\Bigl(\frac1t\Bigr) \Bigl(\frac{d}{dt}\Bigr)^2 \+ \Bigl(4 t^4 F\Bigl(\frac1t\Bigr) - t^3 F'\Bigl(\frac1t\Bigr) \Bigr) \frac{d}{dt}
+  2 t^3 F\Bigl(\frac1t\Bigr) - t^2 F'\Bigl(\frac1t\Bigr) + 1 - \b t \\
& \= t G(t) \Bigl(\frac{d}{dt}\Bigr)^2 + t G'(t) \frac{d}{dt} + t H(t) \\
\eal\]
with
\[\bal
& G(t) \= t + \a_2 t^2 + \a_1 t^3 + \a_0 t^4 \\
& H(t) \= -\b + \a_1 t + 2 \a_0 t^2  
\eal\]
With the notation $D = t \frac{d}{dt}$ we can further rewrite it as
\be{D2}\bal
(1 + \a_2 t + \a_1 t^2 + \a_0 t^3) (D^2-D) &\+ (1 + 2\a_2 t + 3 \a_1 t^2 + 4 \a_0 t^3)D \\
& -\b t + \a_1 t^2 + 2 \a_0 t^3 \\
\= D^2 \+ t (\a_2 D^2 + \a_2 D - \b) &\+ \a_1 t^2 (D+1)^2 + \a_0 t^3 (D+1)(D+2) 
\eal\ee

Notice that putting $\a_0=0$ we obtain precisely operator~\eqref{D2zag} with $A=-\a_2$, $B=\a_1$ and $\l=\b$. 

\section{Modular  equations D2}\label{sec:D2modular}

Recall that a D2 differential equation depends on 4 parameters $(\a_2,\a_1,\a_0,\b)$. It determines a local system of rank~2 over the base
\[
\mathbb{P}^1(\Co) \setminus \{ \infty,\, \text{the roots of } z^3+\a_2 z^2 +\a_1 z+\a_0 \}\,.
\]
Consider also the basis in the space of solutions of D2 near $t=0$ which consists on normalized analytic and logarithmic solutions:
\[\bal
\phi_0(t) &\= 1 + \b t + \Bigl(-\frac12\a_2 \b + \frac14 \b^2 - \frac14\a_1\Bigr)t^2 + \dots \\
\phi_1(t) &\= \log t \, \phi_0(t) \+ \Bigl(-\a_2 - 2\b\Bigr)t + \Bigl(\frac12\a_2^2 + \frac12\a_2\b - \frac34\b^2 - \frac14\a_1\Bigr) t^2 + \dots
\eal\] 

\begin{definition}  We say that an equation D2 with parameters
$(\a_2,\a_1,\a_0, \beta) \in \Q^4$ is modular if the analytic continuation of
\[
\tau \= \frac1{2 \pi i} \frac{\phi_1(t)}{\phi_0(t)}
\]
gives uniformization of the base by the upper halfplane with the group of deck transformations being a congruence subgroup of $\rm{SL}(2,\Z)$ and the function $\tau \mapsto \phi_0(t(\tau))$ is a modular form of weight~1.  
\end{definition}

In this case  $t(\tau)$ is a modular function whose $q$-expansion can be written explicitly. Indeed, inverting the series
\[
q = \exp\Bigl( \frac{\phi_1(t)}{\phi_0(t)}\Bigr) \= t \+ (- \a_2 - 2 \b)t^2 + (\a_2^2 + \frac72 \a_2 \b + \frac{13}4 \b^2 - \frac14 \a_1)t^3 + \dots  
\]
one gets
\[
t \=  q \+ (\a_2 + 2 \b)q^2 + (\a_2^2 + \frac92 \a_2 \b + \frac{19}4 \b^2 + \frac14 \a_1)q^3 + \dots  
\]
Further, substituting this expansion into $\phi_0(t)$ one obtains 
\[
f \= \phi_0(t(\tau)) \= 1 \+  \b q \+ (\frac12 \a_2\b + \frac94 \b^2 - \frac14 \a_1) q^2 \+ \dots 
\]
This must be a modular form of weight~1. 

\bigskip

Put $Q=q^{\frac12}$ and consider the series
\be{cndef}
t^{\frac12}  \sqrt{1+\a_2 t +\a_1 t^2+\a_0 t^3} \, \phi_0(t)^2 \= \sum_{n=1}^{\infty} c_n Q^n  
\ee
whose coefficients $c_n=c_n(\vec{\a},\b)$ can be determined explicitly as follows. One writes 
\[
Q \= \exp\Bigl( \frac12 \frac{\phi_1(t)}{\phi_0(t)}\Bigr) \= t^{\frac12} \Bigl(1 + (-\frac12 \a_2 - \b)t + (\frac38 \a_2^2 + \frac54 \a_2 \b + \frac98 \b^2 - \frac18 \a_1)t^2 + \dots \Bigr)  
\]
and inverts this series in order to get 
\[
t^{\frac12} \= Q \Bigl(1 + (\frac12 \a_2 + \b)Q^2 + (\frac38 \a_2^2 + \frac74 \a_2 \b + \frac{15}8 \b^2 + \frac18 \a_1)Q^4 + \dots \Bigr)
\] 
which can be substituted into the left-hand side of~\eqref{cndef}. We have
\be{cnformulas}\bal
& c_1 \= 1 \\
& c_2 \= 0 \\
& c_3 \= \a_2 + 3 \b \\
& c_4 \= 0 \\
& c_5 \= \a_2^2 + \frac{25}{4} \a_2 \b + \frac{75}{8} \b^2 + \frac18 \a_1 \\
& \dots
\eal\ee
It is not hard to see that all even coefficients in fact vanish. 

\begin{theorem} \label{LfunD2} 
Assume one is given a nondegenerate modular D2  with parameters $\a_2,\a_1,\a_0,\b \in \Q$ . Consider the  weight~2 modular form  $\sum_{n=1}^{\infty} c_n q^n$ with $c_n$  determined from the expansion~\eqref{cndef}. If in addition it is a newform then
\[
L(s) \= \sum_{n=1}^{\infty} \frac{c_n}{n^s}
\]
is the L-function of the elliptic curve
\be{ellcurve}
y^2 = z^3 +\a_2 z^2 +\a_1 z+\a_0 \,. 
\ee
\end{theorem}
By being a newform we
mean that the modular form belongs to the subspace on
 newforms of certain level.
 We do not require it to be a Hecke eigenform a priori; rather,   the Hecke--eigen 
property is a consequence of the theorem. In particular, we have the following

\begin{corollary}
If $\sum_{n=1}^{\infty} c_n q^n$ is a newform, then its coefficients~\eqref{cnformulas} are multiplicative, i.e.
\be{multrel}
c_{mn}(\vec{\a},\b) \= c_{m}(\vec{\a},\b) \, \cdot \, c_{n}(\vec{\a},\b) 
\ee
as soon as $m$ and $n$ are coprime.
\end{corollary}

We will solve equations~\eqref{multrel} with respect to the parameters $\a_2,\a_1,\a_0,\b$ later in this section. It appears that modulo a certain transformation which preserves both the sequence $\{ c_n; n \ge 1\}$ and the $L$-function of~\eqref{ellcurve} there are  finitely many cases. 

\bigskip

The proof of Theorem~\ref{LfunD2} will rely on the following result.

\bigskip

{\bf Theorem} (Atkin $\&$ Swinnerton-Dyer congruences, Theorem~4 in~\cite{ASD}) 
\emph{
Let $p \ne 2,3$, and let $y^2=z^3+Bz+C$ be an elliptic curve over  $\mathbb{Z}_p$ with good reduction. Choose a local parameter at $0$ so that  $z = \xi^{-2} + \sum_{n=-1}^{\infty} d_n \xi^n$ and $y=\xi^{-3}+\dots$ are the respective expansions , and write
\[
- \frac12 \frac{dz}{y} \= \Bigl(\sum_{n=1}^{\infty} c_n \xi^n \Bigr) \frac{d\xi}{\xi}\,.
\]
If $B,C,d_n,c_n$ are $p$-adic integers, then
\be{ASDcong}
c_{np} \- a_p \, c_{n} \+ p \, c_{\frac{n}p} \; \equiv\;  0 \mod p^{\ord_p(n)+1}  
\ee
where
\[
a_p \= - \sum_{m=0}^{p-1} \Bigl( \frac{m^3+Bm+C}{p}\Bigr)\,.
\]
}

Notice that this theorem can be applied to an elliptic curve defined over $\Q$ with  good reduction at $p$  as soon as the coefficients $B$, $C$, $d_n$ and $c_n$ do not contain $p$ in their denominators. Moreover, $a_p = p+1-\#E({\mathbb F}_p)$ is then the $p$-th coefficient of the $L$-function of this elliptic curve.

\begin{proof} [Proof of Theorem~\ref{LfunD2}.] Let $a_n$, $n \ge 1$ be the  coefficients of the $L$-function $L(s)=\sum_n \frac{a_n}{n^s}$ of the elliptic curve~\eqref{ellcurve}. One can check that
\[
\phi_1'(t) \phi_0(t) - \phi_1(t) \phi_0'(t) \= \frac1{t(1+\a_2 t+\a_1 t^2 + a_0 t^3)}\=  t^{-4} F\Bigl(\frac1t\Bigr)^{-1}\,,
\]
hence
\[\bal
& \sum_{n=1}^{\infty} c_n(\vec{\a},\b) Q^n \; \frac{dQ}{Q} \= t^2 F\Bigl(\frac1t\Bigr)^{\frac12} \phi_0(t)^2 \; d \Bigl( \frac12 \frac{\phi_1(t)}{\phi_0(t)}\Bigr) \\
&\= \frac12 t^2 F\Bigl(\frac1t\Bigr)^{\frac12} \Bigl(\phi_1'(t) \phi_0(t) - \phi_1(t) \phi_0'(t)\Bigr) \, dt \= \frac12 t^{-2} F\Bigl(\frac1t\Bigr)^{-\frac12} dt \= -\frac12 \frac{dz}{y}
\eal\]
where we substitute $z = 1/t$, $y^2=F(z)$. This is a holomorphic differential on the curve~\eqref{ellcurve}, and since $Q \sim t^{\frac12}$ for small $t$ we conclude that $Q$ is a local parameter on the curve near the origin. Moreover, $z \sim Q^{-2}$ and $y \sim Q^{-3}$ and therefore the theorem of Atkin and Swinnerton-Dyer would be applicable for every prime $p$ not dividing the conductor of the curve as soon as all $c_n$ and $d_n$ defined from the expansion $z = \frac1{t} = Q^{-2} + \sum_{n=-1}^{\infty} d_n Q^n$ do not contain $p$ in their denominators. First we show that this is indeed the case for all but finitely many primes $p$ using the assumption of modularity.

We have $Q(t(\tau)) = q^{\frac12}$. Looking at~\eqref{cndef} we see that $\sum_{n=1}^{\infty} c_n q^n$ is the $q$-expansion of the modular form $\Bigl[t^{\frac12}  \sqrt{1+\a_2 t +\a_1 t^2+\a_0 t^3} f^2 \Bigr] (2 \tau)$ of weight~2. It follows that possibly after  multiplication by an integer all $c_n$ become integers simultaneously. The same holds for $d_n$ since $z(\tau)$ is a modular function. Therefore for all but finitely many prime numbers $p$ we have congruences~\eqref{ASDcong}. Another consequence of modularity of $\sum_{n=1}^{\infty} c_n q^n$ is that
\be{cnass}
c_n = o(n)\,, \qquad n \to \infty\,.
\ee
Our next step is to show that~\eqref{ASDcong} together with~\eqref{cnass} imply that $c_n=a_n$ for all $n$ not divisible by a finite set of primes. Since $a_n = o(n^{\frac12 + \e})$ for any $\e > 0$ and $c_n=o(n)$ there is a number $N$ such that 
\[
\Big|\frac{c_n}{n} \Big| < \frac12\,, \quad \Big|\frac{a_n}{n} \Big| < \frac12
\]
for all $n > N$. Obviously we can assume that~\eqref{ASDcong} is true for all $p>N$ increasing $N$ if necessary. From~\eqref{ASDcong} with $n=1$ we get
\[
c_p \equiv a_p \mod p\,.  
\]
Since for all $p>N$ also $|c_p|, |a_p| < \frac{p}2$ we conclude that $c_p=a_p$. Suppose $p>N$ and we have proved that $c_{p^m}=a_{p^m}$ for all $m\le M$. Since
\[\bal
& a_{p^{M+1}} \- a_p a_{p^M} + p a_{p^{M-1}} \=0\\
& c_{p^{M+1}} \- a_p c_{p^M} + p c_{p^{M-1}} \equiv 0 \mod p^{M+1}\\
\eal\]
we conclude that
\[
c_{p^{M+1}} \equiv a_{p^{M+1}} \mod p^{M+1}\,.
\] 
Therefore $c_{p^{M+1}}=a_{p^{M+1}}$ because $p^{M+1}>N$ again. It follows now by induction that $c_{p^m}=a_{p^m}$ for all $m \ge 0$ and $p>N$. Our next step is to show that $c_n=a_n$ for all $n$ not divisible by finite number of primes $p \le N$. Let $n$ be such a number and suppose that for every proper divisor $n'|n$ we have already proved that $a_{n'}=c_{n'}$. By~\eqref{ASDcong} with a prime divisor $p|n$ and  $n'=\frac{n}{p}$ instead of $n$ we have
\[
c_n \- a_p c_{n'} + p c_{\frac{n'}{p}} \= c_n \- a_n \equiv 0 \mod p^m 
\]
where $m=\ord_p(n)$. Since this is true for every prime dividing $n$ we conclude that $c_n \equiv a_n \mod n$, and therefore $c_n=a_n$ again by our estimate.

Consider both newforms $\sum_n c_n q^n$ and $\sum_n a_n q^n$ on the intersection of the corresponding congruence subgroups which is again a congruence subgroup. Since $c_n=a_n$ for all $n$ not divisible by primes form a certain finite set, it follows that both forms have the same eigenvalues for infinitely many Hecke operators. Therefore by multiplicity one theorem these forms are just equal, and our theorem is proved.
\end{proof}

Now we can  substitute the polynomials~\eqref{cnformulas} into the multiplicativity relations~\eqref{multrel} and solve the resulting equations. We do not expect finitely many solutions because the shifts~\eqref{shiftDN} preserve modularity. Under the shift $A \mapsto A'=A+\e$ the parameters become
\[
(\a_2',\a_1',\a_0',\b') \= (\a_2-3 \e, 3\e^2 - 2\a_2\e + \a_1,-\e^3 + \a_2\e^2 - \a_1\e + \a_0,\b+\e) \,.
\]
Zagier's choice $\a_0=0$ is not natural from this point of view because one can make $\a_0=0$ only if $F(z)$ has a rational root. It is more natural to choose the equation with $\b=0$ as a unique representative of the orbit of the shifts. 

Solving the first few relations with Gr\"{o}bner bases ( we used computer algebra system~\cite{MAGMA}) 
\[\bal
& c_6\=c_2 \cdot c_3\,,\quad c_{10}\=c_2 \cdot c_5\,,\quad c_{12}\=c_4 \cdot c_3\,,\quad c_{14}\=c_2 \cdot c_7\,,\\
& c_{15}\=c_3 \cdot c_5\,,\quad c_{18}\=c_2 \cdot c_9\,,\quad c_{21}\=c_3 \cdot c_7\,,\quad c_{22}\=c_2 \cdot c_{11}\\
\eal\]
we obtain 8 points $(\a_2,\a_1,\a_0)$ plus two one-parametric families $(0,0,\a_0)$ and $(0,\a_1,0)$. In order to show that there are actually finitely many cases in these families we used more relations by considering about 200 further coefficients. The results are given in the table below.

\bigskip

\begin{center}
\begin{tabular}{ccc|l|l}
$\a_2$&$\a_1$& $\a_0$ & $F(z)$ & $(A,B,\l)$\\
\hline
1&0&0&$z^2(z+1)$&$(-1,0,0),(2,1,1)$\\
-1&0&0&$z^2(z-1)$&$(1,0,0),(-2,1,-1)$\\
-4&-80&-192&$(z-12)(z+4)^2$&$(-32,256,-12),(16,0,4)$\\
4&-80&192&$(z+12)(z-4)^2$&$(32,256,12),(-16,0,-4)$\\
\hline
-2&-40&-75&$(z+3)(z^2-5z-25)$&$(11,-1,3)$\\
2&-40&75&$(z-3)(z^2+5z-25)$&$(-11,-1,-3)$\\
-1&-24&-36&$(z-6)(z+2)(z+3)$&$(-17,72,-6),(7,-8,2),(10,9,3)$\\
1&-24&36&$(z+6)(z-2)(z-3)$&$(17,72,6),(-7,-8,-2),(-10,9,-3)$\\
0& $16 \zeta_4$ &0&$z(z^2 - 16 \zeta_4)$&$(0,\pm16,0),(12,32,4),(-12,32,-4)$\\
0&0& $27 \zeta_6$ &$z^3-27 \zeta_6$&$(9,27,3),(-9,27,-3)$\\
\end{tabular}
\end{center}
\bigskip

\noindent
The first three columns contain all solutions of a few first multiplicativity equations. 
The fourth column shows the roots of the respective polynomial $F(z)=z^3+\a_2 z^2 + \a_1 z+\a_0$. Polynomials in the first 4 rows appear to have multiple roots meaning that the respective differential operator is degenerate. In the last column we shift the differential operator by various roots of $F(z)$ in order to obtain operators with $\a_0=0$, the respective values of the parameters $(A,B,\lambda)=(-\a_2,\a_1,\b)$ being listed. The last 6 rows give us D2 equations, and shifting by various rational roots we obtain precisely Zagier's table. One can easily check at this point that in each case the statement of Theorem~\ref{LfunD2} holds, and therefore the respective $(\a_2,\a_1,\a_0)$ indeed solve all multiplicativity equations. The triples corresponding to the degenerate differential equations from the first four rows can be found in~\cite{Zag} as $\#1,\#3, \#19$ and $\#11$. On the other hand, the degenerate triples $\# 14, \#20$ and $\#25$ are also modular but do not appear on our list.

\section{Differential equations of type D3}\label{sec:D3}

Our goal in this section is to write the generic form of a D3 equation by making exactly the same steps as in Section~\ref{sec:D2} but now with $N=3$. We get

\[\bal
&\L_{A,\infty} \= \det\,_\mathrm{right} \, \begin{pmatrix} 
(z-a_{00})\frac{d}{dz} & -a_{01}\bigl(\frac{d}{dz}\bigr)^2 & -a_{02}\bigl(\frac{d}{dz}\bigr)^3 & -a_{03}\bigl(\frac{d}{dz}\bigr)^4 \\
-1 & (z-a_{11})\frac{d}{dz} & -a_{12}\bigl(\frac{d}{dz}\bigr)^2 & -a_{02}\bigl(\frac{d}{dz}\bigr)^3 \\
0 & -1 & (z-a_{11})\frac{d}{dz} & -a_{01}\bigl(\frac{d}{dz}\bigr)^2\\
0 & 0 & -1 & (z-a_{00})\frac{d}{dz}\\
\end{pmatrix} \; \bigl(\frac{d}{dz}\bigr)^{-1}\\
&\= F(z) \bigl(\frac{d}{dz}\bigr)^3 \+ \frac32 F'(z) \bigl(\frac{d}{dz}\bigr)^2 \+ \Bigl(\frac12 F''(z) + G(z) \Bigr) \frac{d}{dz} + \frac12 G'(z)
\eal\]
with
\[\bal
F(z) &\= \det\bigl(z-A\bigr) \= z^4 \+ \a_3 z^3 \+ \a_2 z^2 \+ \a_1 z \+ \a_0\\
&\a_3 \= -2 a_{11} - 2 a_{00}\\ 
&\a_2 \= 4 a_{00} a_{11} + a_{00}^2 + a_{11}^2 - 2 a_{01} - a_{12}\\
&\a_1 \= -2 a_{02} -2a_{11}a_{00}^2 + 2 a_{00}(a_{01} -a_{11}^2 + a_{12}) + 2 a_{11}a_{01}\\
&\a_0 \= -a_{03} + 2 a_{00}a_{02} + (a_{11}^2 - a_{12})a_{00}^2 - 2 a_{11}a_{01}a_{00} + a_{01}^2\\
G(z) &\= z^2 \+ \b_1 z + \b_0\\ 
&\b_1 \= -2 a_{00} \\
&\b_0 \= 2 a_{00} a_{11} - a_{11}^2 - 2 a_{01} + a_{12}\\
\eal\]

Recall that  a $D3_{\infty,1}$ differential equation is called non--degenerate whenever the roots of $F(z)$ are distinct. Notice that our order $3$  differential operator is the symmetric square of the order $2$ operator
\[
F(z) \bigl(\frac{d}{dz}\bigr)^2 \+ \frac12 F'(z) \frac{d}{dz} \+ \frac14 G(z)\,.
\]

\bigskip

 We also compute the generic $D3_{0,0}$. We have
\[\bal
&-F\Bigl(\frac1t\Bigr) \Bigl(-t^2 \frac{d}{dt}\Bigr)^3 t \- \frac32 F'\Bigl(\frac1t\Bigr) \Bigl(-t^2 \frac{d}{dt}\Bigr)^2 t \- 
\Bigl( \frac12 F''\Bigl(\frac1t\Bigr) + G\Bigl(\frac1t\Bigr)\Bigr) \Bigl(-t^2 \frac{d}{dt}\Bigr)t \- \frac12 G'\Bigl(\frac1t\Bigr) t\\
& \= t \Bigl[ H(t) \bigl(\frac{d}{dt}\bigr)^3 \+ \frac32 H'(t) \bigl(\frac{d}{dt}\bigr)^2 \+ \Bigl(\frac12 H''(t) + U(t) \Bigr) \frac{d}{dt} + \frac12 U'(t) \Bigr] 
\eal\]
with
\[\bal
&H(t) \= t^6 F \Bigl(\frac1t\Bigr) \= t^2 + \a_3 t^3 + \a_2 t^4 + \a_1 t^5 + \a_0 t^6 \\
&U(t) \= t^2 G \Bigl(\frac1t\Bigr) \+ 3 t^4 F \Bigl(\frac1t\Bigr) \- t^3 F' \Bigl(\frac1t\Bigr)  \= 3 \a_0 t^4 + 2 \a_1 t^3 + (\a_2 + \b_0) t^2 + \b_1 t   
\eal\]

Finally, this differential operator can be written as
\be{D3}\bal
D^3 &\+ t\bigl(D + \frac12\bigr) \bigl(\a_3 (D^2 + D) \+ \b_1\bigr) \\
&\+ t^2 (D+1) \bigl( \a_2 (D+1)^2 \+ \b_0  \bigr) \\
&\+ \a_1 t^3 (D+2)\bigl(D+\frac32\bigr)(D+1)\\
&\+ a_0 t^4 (D+3)(D+2)(D+1) 
\eal\ee

\section{ Nondegenerate modular equations D3}

In this section we will prove the analog of Theorem~\ref{LfunD2} for D3 equations. In order to state  it we first associate to such an equation an appropriate elliptic curve. Recall that $F(z)=\det\bigl(z-A\bigr)$ has distinct roots, so the discriminant of $F$ is nonzero. For a $D3$ equation we have $F(z) \= z^4 \+ \a_3 z^3 \+ \a_2 z^2 \+ \a_1 z \+ \a_0$. Consider the curve
\be{C}
w^2 \= z^4 \+ \a_3 z^3 \+ \a_2 z^2 \+ \a_1 z \+ \a_0.
\ee
 Put it into the Weierstrass form.

\begin{definition}\label{spec_ell} The spectral elliptic curve of a D3 equation is 
\be{spec_ell_D3}\bal
y^2 \= x^3 \+ \Big( \a_1\a_3 - & \frac13 \a_2^2 - 4 \a_0 \Bigr) x \\
& \+ \Bigl( \a_0\a_3^2-\frac13\a_1\a_2\a_3+\frac2{27}\a_2^3-\frac83\a_0\a_2+\a_1^2 \Bigr) \,. 
\eal\ee
\end{definition}

\noindent
It is indeed an elliptic curve because the discriminant of the cubic polynomial in the right-hand side is equal to the discriminant of $F(z)$ (as a function of $\a_i$), and therefore the right-hand side has 3 distinct roots. 

\begin{lemma}\label{curves_lemma} Curves~\eqref{C} and~\eqref{spec_ell_D3} are birational over the splitting field of the polynomial $F(z)$. Moreover, the holomorphic differential $\frac{dx}{y}$ on the spectral elliptic curve transforms into $-\frac{dz}{w}$ on~\eqref{C} under this birational equivalence.  
\end{lemma}
\begin{proof}
Let $F(z_0)=0$. Then 
\[
w^2 \= (z-z_0)^4 \+ \tilde{\a}_3 (z-z_0)^3 \+ \tilde{\a}_2 (z-z_0)^2 \+ \tilde{\a}_1 (z-z_0)
\]
with $\tilde{\a}_i=\frac1{i!}F^{(i)}(z_0)$, and 
\[
\Bigl(\frac{\tilde{\a}_1 \, w}{(z-z_0)^2}\Bigr)^2 \= \tilde{\a}_1^2 \+ \tilde{\a}_1\tilde{\a}_3 \Bigl(\frac{\tilde{\a}_1 }{z-z_0}\Bigr) \+ \tilde{\a}_2 \Bigl(\frac{\tilde{\a}_1 }{z-z_0}\Bigr)^2 \+ \Bigl(\frac{\tilde{\a}_1 }{z-z_0}\Bigr)^3 \,.
\]
Hence the variables
\be{xyzw}
x \= \frac{\tilde{\a}_1 }{z-z_0}+\frac{\tilde{\a}_2}{3}\,,\quad y \= \frac{\tilde{\a}_1 \, w}{(z-z_0)^2}
\ee
satisfy $y^2=x^3+Bx+C$ with
\[\bal
B &\= \tilde{\a}_1\tilde{\a}_3 - \frac{\tilde{\a}_2^2}{3} \= \a_1\a_3 -  \frac13 \a_2^2 - 4 \a_0\,,\\
C &\= \tilde{\a}_1^2 + \frac2{27}\tilde{\a}_2^3 - \frac13  \tilde{\a}_1\tilde{\a}_2\tilde{\a}_3 \= \a_0\a_3^2-\frac13\a_1\a_2\a_3+\frac2{27}\a_2^3-\frac83\a_0\a_2+\a_1^2\,.
\eal\]
The equality of the differentials follows immediately.
\end{proof}

\bigskip

Consider the normalized analytic and logarithmic solutions of D3 near $t=0$:
\be{frbas}\bal
\phi_0(t)  &\= 1  \- \frac12 \beta_1 t \+ \Bigl(\frac3{16}\alpha_3\beta_1+\frac3{32}\beta_1^2-\frac18\alpha_2 -\frac18\beta_0\Bigr) t^2 + \dots\\
\phi_1(t)  &\= \log t \, \phi_0(t) + \psi(t) \= \log t \, \phi_0(t) \+ \Bigl(-\frac12\alpha_3+\frac12 \beta_1\Bigr)t \+ \dots\\\
\eal\ee

Again, the 6 parameters naturally split into two groups. Parameters $\a_3, \a_2,\a_1,\a_0$ determine the base 
\[
\mathbb{P}^1(\Co) \setminus \{ \infty,\, \text{the roots of } z^4+\a_3 z^3+\a_2 z^2 +\a_1 z+\a_0 \}
\]
and also the spectral elliptic curve~\eqref{spec_ell_D3}. Of course the solutions~\eqref{frbas} and the respective local system of rank~3 depend also on the remaining parameters $\b_1$ and $\b_0$.   

\begin{definition} We say that an equation D3 with  $\a_3,\a_2,\a_1,\a_0, \b_1, \b_0 \in \mathbb{Q}$  is modular if the analytic continuation of
\be{tau_map}
\tau \= \frac1{2 \pi i} \frac{\phi_1(t)}{\phi_0(t)}
\ee
gives uniformization of the base by the upper halfplane with the group of deck transformations being a congruence subgroup of $\rm{SL}(2,\Z)$ and the function $\tau \mapsto \phi_0(t(\tau))$ is a modular form of weight~2. 
\end{definition}

Consider the power series
\[
q(t) \= \exp \Bigl(\frac{\phi_1(t)}{\phi_0(t)} \Bigr) \= t \exp\Bigl(\frac{\psi(t)}{\phi_0(t)} \Bigr) \= t + \Bigl(-\frac12\a_3 + \frac12 \b_1\Bigr)t^2  + \dots 
\] 
We can invert this expansion in order to write $t$ as a power series in $q$, and for accessory values of parameters this must be then the $q$-expansion of a modular function. Analogously, $\phi_0(t)$ written as a power series in $q$ will be the $q$-expansion of a modular form of weight~2 since the local system is of rank~3. We will be specifically interested in the coefficients of the modular form
\be{fgl}
t \, \phi_0(t) \= \sum_{n=1}^{\infty} c_n q^n 
\ee
which can be computed explicitly as polynomials in the initial parameters 
\be{cnformulasD3}\bal
& c_1 \= 1 \\
& c_2 \= \frac12 \a_3 - \b_1 \\
& c_3 \= \frac3{16}\a_3^2 - \frac{27}{32}\a_3 \b_1 + \frac{57}{64}\b_1^2 + \frac1{16}\a_2 - \frac3{16}\b_0 \\
& \dots
\eal\ee

\bigskip 

\begin{theorem}\label{Lfun} Assume we are given a non—degenerate modular D3. If the modular form~\eqref{fgl} of weight~2  is  a newform then
\[
\sum_{n=1}^{\infty} \frac{c_n}{n^s} 
\]
is the L-function of the spectral elliptic curve~\eqref{spec_ell_D3}.
\end{theorem}

\bigskip

\begin{proof} 
We will use the same method as for D2. Let $a_n, n \ge 1$ be the coefficients of the $L$-function $L(s)=\sum_n a_n \, n^{-s}$  of the spectral elliptic curve~\eqref{spec_ell_D3}.  One can check that 
\be{holdif}\bal
t \, \phi_0(t) \, \frac{d q}{q} &\= t \, \phi_0(t) \, d \log q \= t \frac{\phi_1'(t) \phi_0(t) - \phi_1(t) \phi_0'(t)}{\phi_0(t)} dt \\
&\= F\Bigl(\frac1t\Bigr)^{-\frac12} t^{-2} \, dt \= -\frac{dz}{w}  
\eal\ee  
where we substitute $z = 1/t$, $w^2=F(z)$. 
Now using the birational transformation~\eqref{xyzw} we see that $t$ (hence also $q$) is a local parameter on the spectral elliptic curve near the point $P=\Bigl(\frac13 \tilde{\a}_2,\tilde{\a}_1\Bigr)$ where $\tilde{\a}_i=\frac1{i!}F^{(i)}(z_0)$ and $z_0$ is a chosen root of $F(z)=0$ as in the proof of Lemma~\ref{curves_lemma}. Since our differential equation is modular, composition of modular uniformization with the birational transform gives a map from a modular curve to the spectral elliptic curve. 
The preimages of both points $ \pm P=\Bigl(\frac13 \tilde{\a}_2,\pm \tilde{\a}_1\Bigr)$ are cusps because $t$ is zero there and this is a cuspidal value as we know. 
Therefore by the Manin-Drinfeld theorem their difference $P-(-P)=2P$ and hence also $P$ is a point of finite order. 
Let us write the spectral curve~\eqref{spec_ell_D3} as $y^2 = x^3 + B x + C$. The differential $- \frac{dz}{w}$ in~\eqref{holdif} transforms to $\frac{dx}{y}$ according to Lemma~\ref{curves_lemma}. 
One can find an isogenous curve $\tilde y^2 = \tilde x^3 + B \tilde x + C$ where $P$ is mapped to the origin and the differential is mapped to $-\frac12\frac{d\tilde x}{\tilde y}$. Then $q$ is a local parameter on the latter curve near the origin, and the $L$-function of this curve is again $L(s)=\sum_n a_n \, n^{-s}$ since isogenous curves have equal $L$-functions. 
One can check that since the expansion $-\frac12\frac{d\tilde x}{\tilde y} \= \Bigl( \sum_{n=1}^{\infty} c_n q^n \Bigr) \frac{dq}{q}$ starts with $c_1=1$ then $\tilde x \sim q^{-2}$. All the coefficients $c_n$ in \eqref{fgl} and $d_n$ in the expansion $\tilde x = q^{-2} + \sum_{n=-1}^{\infty} d_n q^n$ do not contain $p$ in denominators for all but finitely many primes $p$ because these are $q$-expansions of a modular form of weight~2 and modular function respectively. The rest of the proof goes exactly like in Theorem~\ref{LfunD2}. Namely, one has Atkin and Swinnerton--Dyer congruences for all but finitely many primes and together with $c_n = o(n)$, which is another consequence of modularity, this implies that $a_n=c_n$ for all $n$ non divisible by a finite set of primes. 
The normalized newforms $\sum_n a_n q^n$ and $\sum_n c_n$ are then in one Hecke-eigenspace and therefore are equal by multiplicity one theorem for the space of newforms.   
\end{proof}

By this theorem, if $\sum_{n=1}^{\infty} c_n q^n$ is a newform then the coefficients $c_n$ are the coefficients of the $L$-function of an elliptic curve, and we have the following consequence. 

\begin{corollary}
If $\sum_{n=1}^{\infty} c_n q^n$ is a newform, its coefficients~\eqref{cnformulasD3} are multiplicative, i.e.
\be{multrel1}
c_{mn}(\vec{\a},\b) \= c_{m}(\vec{\a},\b) \, \cdot \, c_{n}(\vec{\a},\b) 
\ee
as soon as $m$ and $n$ are coprime.
\end{corollary}

Let us solve equations~\eqref{multrel1} for the parameters $(\vec{\a},\b)$.

\begin{lemma}\label{shift_lemma_D3} A D3 equation has the following properties with respect to the shift~\eqref{shiftDN}:   
\begin{itemize}
\item[(i)] the parameters change according to the rule
\[\bal
(&\a_3',\a_2',\a_1',\a_0')\\
&\=(\a_3 - 4\e,\a_2 -3\e\a_3 + 6\e^2,\a_1 - 2 \e\a_2 + 3 \e^2\a_3 - 4\e^3,\a_0-\e\a_1 +\e^2 \a_2 -\e^3\a_3 +\e^4) \\
(&\b_1',\b_0') \= (\b_1-2\e,\b_0-\b_1 \e+\e^2)
\eal\]
\item[(ii)] the coefficients of the spectral elliptic curve~\eqref{spec_ell_D3} do not change
\item[(iii)] all $c_n$ in~\eqref{cnformulasD3} do not change
\item[(iv)] modular differential equations transform into modular ones (assuming $\e \in \Q$) 
\end{itemize}
\end{lemma}

\begin{proof}
Indeed, since for $A'=A+\e$ one has  $\L_{A',\infty}(z) \= \L_{A,\infty}(z-\e)$ the formulas in~(i) follow from 
\[\bal
z^4 + \a_3' z^3 & + \a_2' z^2 + \a_1' z + \a_0' \\
& \= (z-\e)^4 + \a_3 (z-\e)^3 + \a_2 (z-\e)^2 \+ \a_1 (z-\e) + \a_0 ,\\
z^2 + \b_1' z & + \b_0' \= (z-\e)^2 + \b_1(z-\e) + \b_0 \,.
\eal\]
Part (ii) follows by a tedious computation. Next, from $\L_{A',0}(t) \= \L_{A,0}\Bigl( \frac{t}{1-\e t}\Bigr) \, (1-\e t)$ we conclude that Frobenius bases must simply transform as $\phi_i'(t) \= (1-\e t)\phi_i\Bigl( \frac{t}{1-\e t}\Bigr)$ for $i=0,1,2$. Therefore the uniformization maps~\eqref{tau_map} differ by the transformation $t \mapsto \frac{t}{1-\e t}$ of the base, from where~(iii) and~(iv) follow immediately.
\end{proof}

It follows from Lemma~\ref{shift_lemma_D3} that it suffices to solve equations~\eqref{multrel1} for a single representative of every orbit of shifts. 
There are two natural choices, $\a_3=0$ and $\b_1=0$. We will use the latter one. It appears that the system of equations~\eqref{multrel1} has finitely many solutions with $\b_1=0$. We will give the complete list later, but first   we list the rational solutions that correspond to nondegenerate D3, i.e. such that the roots of $F(z)$ are all distinct. There are exactly 18 of them. We list the $\a_i$'s In the table below.  These determine the spectral elliptic curve~\eqref{spec_ell_D3} which we denote by $\mathcal{E}$. Then we give its $j$-invariant, its level $N$ and the newform $g(\tau)$ of level $N$ 
whose Mellin transform is the $L$-function of $\mathcal{E}$. 
We give the value of $\b_0$ in the last column.

\begin{center}
\begin{tabular}{cccc|c|c|c|c}
$\a_3$&$\a_2$&$\a_1$& $\a_0$ &$j(\mathcal{E})$ &$N(\mathcal{E})$&  $g_{\mathcal{E}}(\tau) $ & $\b_0$\\
\hline
&&&&&&\\
0&-44&0&-16&$-\frac{20720464}{15625}$& 20& ${\bf 2^2 \cdot 10^2}$&-4  \\
0&44&0&-16&& 80 &${\bf 4^6 \cdot 20^6 / 2^2 \cdot 8^2 \cdot 10^2 \cdot 40^2}$ &4 \\
0&-28&0&-128&$\frac{207646}{6561}$& 24 & ${\bf 2\cdot 4\cdot 6\cdot 12}$&-4 \\
0&28&0&-128&& 48 & ${\bf 4^4 \cdot 12^4 / 2\cdot6\cdot8\cdot24}$&4 \\
0&-40&0&144&$\frac{35152}{9}$& 24 &${\bf 2\cdot 4\cdot 6\cdot 12}$&-8  \\
0&40&0&144& & 48 & ${\bf 4^4 \cdot 12^4 / 2\cdot6\cdot8\cdot24}$&8\\
0&-68&0&1152& $\frac{3065617154}{9}$ & 48 & ${\bf 4^4 \cdot 12^4 / 2\cdot6\cdot8\cdot24}$&-28 \\
0&68&0&1152& &24 &${\bf 2\cdot 4\cdot 6\cdot 12}$&28  \\
0&-48&0&512& 287496& 32 & ${\bf 4^2\cdot 8^2}$& -16 \\
0&48&0&512& &32 & ${\bf 4^2\cdot 8^2}$&16 \\
0&0&0&-256& 1728 & 32 & ${\bf 4^2\cdot 8^2}$&0  \\
0&0&0&256& & 64 & ${\bf 8^8 / 4^2 \cdot 16^2}$&0  \\
0&-36&0&432& 54000&144 & ${\bf 12^{12} / 6^4\cdot 24^4}$&-12  \\
0&36&0&432& & 36 & ${\bf 6^4}$&12  \\
-4&-88&-300&-304 & $-\frac{122023936}{161051}$&11&${\bf 1^2 \cdot 11^2 }$&-8\\
0&0&-108&0 &0&27&${\bf 3^2 \cdot 9^2}$&0\\
-2&-43&-156&-216&$\frac{4733169839}{3515625}$& 15& $ {\bf1 \cdot 3 \cdot 5 \cdot 15}$&-5\\
&&&&&&\\
-2&-59&-136&-80&$\frac{4956477625}{941192}$ &14&${\bf 1 \cdot 2 \cdot 7 \cdot 14}$ &-5 \\
\end{tabular}
\end{center}

In the proof of the Theorem~\ref{Lfun} we have constructed the point $P=\Bigl(\frac13 \tilde{\a}_2,\tilde{\a}_1\Bigr)$ of finite order on the spectral elliptic curve. The order of $P$ is 4 in the first 14 cases and for the last 4 rows it is 5,3,8 and 3 respectively. 

\section{All solutions of the multiplicativity equations for D3}

The goal of this section is to list all the solutions to the multiplicativity equations. We have obtained them via Gr\"{o}bner bases with the aid of the computer algebra system \cite{MAGMA}, doing computations over several finite fields and lifting solutions afterwards. Apart from the non--degenerate cases which we listed in the previous section, there are solutions defined over number fields and also solutions with degenerate $F(z)$. In order to list them in an efficient way we consider the twists $A \mapsto A'' = \Bigl( \l^{j-i+1} a_{ij} \Bigr)$ which lead to the simple variable change in the differential equation $\L_{A'',0}(t) \= \L_{A,0}( \l t)$.

\begin{lemma}\label{twist_lemma_D3} 
A D3 equation has the following properties with respect to the twist $\L_{A'',0}(t) \= \L_{A,0}( \l t)$: 
\begin{itemize}
\item[(i)] the parameters transform according to the rule
\[\bal
\vec{\a}'' \= (\l \alpha_3, \l^2 \a_2, \l^3 \a_1, \l^4 \a_0)\,, \quad \vec{\b}'' \= (\l \b_1, \l^2 \b_0)   
\eal\]
\item[(ii)] the spectral curve transforms via 
\[
y^2 \= x^3 \+ \l^4 B x \+ \l^6 C 
\]
\item[(iii)] the coefficients $c_n(\vec{\a}, \vec{\b})$ transform via $c_n'' \= \l^{n-1} c_n$
\item[(iv)] for the function $\tau(t) \= \frac1{2 \pi i} \frac{\phi_1(t)}{\phi_0(t)}$, one has
\[
\tau'' \= \tau \- \frac{\log \l}{2 \pi i}\,,
\] 
\end{itemize} 
\end{lemma}
  
The proof is straightforward. 
The way the solutions transform under the twists  described in ~(iii) and~(iv) of Lemma~\ref{twist_lemma_D3}   shows that only finitely many twists are possible for a given D3  that preserve the multiplicativity property, and the only twists possible 
are those by roots of unity.
We give the complete list of solutions to~\eqref{multrel1} below. We list only one representative in every family of twists  and give all possible twists in the right-most column. We start with nondegenerate cases.  These have already been given  in the previous section up to the twists by roots of unity. 

\begin{center}
\begin{tabular}{cccc|c|c|c|c|c}
$\a_3$&$\a_2$&$\a_1$& $\a_0$ &$j(\mathcal{E})$ &$N(\mathcal{E})$&  $g_{\mathcal{E}}(\tau) $ & $\b_0$& twists \\
\hline
&&&&&&&\\
0&-44&0&-16&$-\frac{20720464}{15625}$& 20& ${\bf 2^2 \cdot 10^2}$&-4 & $\l^2 = \pm1$ \\
0&-28&0&-128&$\frac{207646}{6561}$& 24 & ${\bf 2\cdot 4\cdot 6\cdot 12}$&-4& $\l^2 = \pm1$ \\
0&-40&0&144&$\frac{35152}{9}$& 24 &${\bf 2\cdot 4\cdot 6\cdot 12}$&-8 & $\l^2 = \pm1$ \\
0&68&0&1152& $\frac{3065617154}{9}$ &24 &${\bf 2\cdot 4\cdot 6\cdot 12}$&28 & $\l^2 = \pm1$\\
0&48&0&512& 287496 &32 & ${\bf 4^2\cdot 8^2}$&16 & $\l^{16} = 1$ \\
0&0&0&-256& 1728 & 32 & ${\bf 4^2\cdot 8^2}$&0 & $\l^{16} = 1$\\
0&36&0&432& 54000& 36 & ${\bf 6^4}$&12 & $\l^{36} = 1$\\
-4&-88&-300&-304 & $-\frac{122023936}{161051}$&11&${\bf 1^2 \cdot 11^2 }$&-8& $\l = \pm1$ \\
0&0&-108&0 &0&27&${\bf 3^2 \cdot 9^2}$&0&$\l^{18}=1$\\
-2&-43&-156&-216&$\frac{4733169839}{3515625}$& 15& $ {\bf1 \cdot 3 \cdot 5 \cdot 15}$&-5& $\l = \pm1$\\
-2&-59&-136&-80&$\frac{4956477625}{941192}$ &14&${\bf 1 \cdot 2 \cdot 7 \cdot 14}$ &-5& $\l = \pm1$\\
\end{tabular}
\end{center}

In addition, there is one more nondegenerate solution over $\Q(\sqrt{5})$

\begin{center}
\begin{tabular}{cccc|c|c|c}
$\a_3$&$\a_2$&$\a_1$& $\a_0$ &$j(\mathcal{E})$ & $\b_0$& twists\\
\hline
&&&&&&\\
0&$22-30\sqrt5$&0&$1000-440\sqrt5$ & $\notin \Q$ &  $18 - 10\sqrt{5}$& $\l^2 = \pm 1$ \\
\end{tabular}
\end{center}

\noindent
which appears to give the same modular form $t(\tau) \phi_0(\tau) = {\bf 2^2 \cdot 10^2}$ as in the first row of the above table.

\bigskip

The solutions of~\eqref{multrel1} with degenerate polynomial $F(z) \= z^4+\a_3 z^3+\a_2 z^2 + \a_1 z + \a_0$ are given in the table below. Remarkably, some of them are still modular.

\begin{center}
\begin{tabular}{ccccc|c|c}
$\a_3$&$\a_2$&$\a_1$& $\a_0$ & $\b_0$ &  $t(\tau) \phi_0(\tau) $ &  twists\\
\hline
&&&&&&\\
0&0&0&0&0 & $q$ & --- \\
0&4&0&0&-4 & $q/(1-q^2)$& $\l^2 = \pm 1$\\
0&-8&0&16&-8&$q/(1-q^2)$ & $\l^2 = \pm 1$ \\
0&128&0&4096&64& ${\bf 2^4\cdot 8^4 / 4^4 }$ & $\l^2 = \pm 1$ \\
0&64&0&0&0& ${\bf 4^8 / 2^4}$ &$\l^2 = \pm 1$\\
4&0&0&0&0&$q/(1-q)^2$& $\l = \pm 1$\\
2&1&0&0&-1&$q/(1-q)$ & $\l = \pm1$\\
-2&-3&0&0&3&$q/(1+q+q^2)$& $\l = \pm1$ \\
-6&-135&-540&-648&-9&${\bf 1^3 \cdot 9^3/ 3^2}$&$\l = \pm1$\\
4&-80&192&0&-16&${\bf 2^4\cdot 6^4 /1^2 \cdot 3^2}$&$\l = \pm1$\\
2&9&-216&432&-9&${\bf 3^3\cdot 6^3 /1 \cdot 2}$&$\l = \pm1$\\
2&-55&-100&1000&-25&${\bf 5^5 /1}$ & $\l = \pm 1$\\
8&-176&768&-1024&-16 & ${\bf 2^{10}\cdot 8^4 /1^4 \cdot 4^6}$ & $\l = \pm 1$\\
\end{tabular}
\end{center}

\bigskip

\section{From D2's to D3's}

Let $\phi_0(t)=1+u_1 t + u_2 t^2 + \dots$ be the solution of the differential equation~\eqref{D2} with $\a_0=0$. As we already mentioned, this is exactly the class of D2's considered by Zagier in~\cite{Zag} where his parameters $A,B,\l$ are our $-\a_2,\a_1,\b$ correspondingly. One then has
\[
(n+1)^2 u_{n+1} + (\a_2 n^2+\a_2 n-\b)u_n + \a_1n^2u_{n-1}=0\,,
\]
and we observe that the sequence $u_n'=\binom{2n}{n}u_n$ satisfies
\[
(n+1)^3 u_{n+1} + 2 (2n+1)(\a_2 n^2+\a_2 n-\b)u_n + 4\a_1(2n+1)(2n-1)nu_{n-1}=0\,.
\]
In other words, $\phi_0'(t)=\sum_{n=0}^{\infty}\binom{2n}{n}u_nt^n$ is a solution of
\[\bal
D^3 &\+ t\bigl(D + \frac12\bigr) \bigl(4 \a_2 (D^2 + D) \- 4\b\bigr) \\
&\+ t^2 (D+1) \bigl( 16\a_1 (D+1)^2 \- 4\a_1  \bigr) \\
\eal\]
which turns out to be a D3 equation  with parameters
\[
(\a'_3,\a'_2,\a'_1,\a'_0)\=(4 \a_2,16\a_1,0,0)\,, \quad(\b_1',\b_0')=(-4\b,-4\a_1)\,.
\]
This equation is degenerate as its symbol has double roots. Passing  to  $\phi_0''(t)=\sum_{n=0}^{\infty}\binom{2n}{n}u_nt^{2n}$, which is a solution of
\[
D^3 \+ t^2 (D+1) \bigl( 4\a_2 (D+1)^2 - 16 \b - 4 \a_2  \bigr) \+ 16 \a_1 t^4 (D+3)(D+2)(D+1) \,,
\]
we thus come to a D3 equation whose parameters are
\[
(\a''_3,\a''_2,\a''_1,\a''_0)\=(0,4\a_2,0,16 a_1)\,, \quad(\b_1'',\b_0'')=(0,- 16 \b - 4 \a_2 )\,.
\]
This D3 is nondegenerate if the initial D2 was nondegenerate. Indeed, the symbol $F''(z)=z^4+4\a_2 z^2+16\a_1$ has 4 distinct roots if and only if $\a_1 \ne 0$ and $\a_2^2\ne 4\a_1$, which are exactly the conditions for the roots of the symbol $F(z)=z^3+\a_2z^2+\a_1z$ of the initial D2 to be all distinct. Therefore we have a map from nondegenerate D2's with $\a_0=0$ to nondegenerate D3's with $\a_3=\a_1=\b_1=0$ given by
\be{D2toD3}
\bal \vec{\a} &= (\a_2,\a_1,0) \\
\vec{\b} &= (\b) \eal
\quad \mapsto \quad 
\bal \vec{\a} &= (0,4\a_2,0,16 \a_1) \\
\vec{\b} &= (0,- 16 \b - 4 \a_2 ) \eal \\
\ee
This map is obviously bijective. The analytic solution at $t=0$ transforms as 
\[ 
\sum_{n=0}^{\infty}u_n\,t^{n}
\quad \mapsto \quad
\sum_{n=0}^{\infty}\binom{2n}{n} \, u_n \, t^{2n}\,.
\]
We find that under the map~\eqref{D2toD3} Zagier's triples all go into D3's from our list of accessory equations and, moreover, they exhaust all D3's in our list with $\a_3=\a_1=0$. There 14 cases on both lists, below we show half of them as we already did in Introduction. Namely, for every D2 listed in the table with the parameters $(\a_2,\a_1,\a_0)$ there is also $(-\a_2,\a_1,-\a_0)$. We denote the spectral elliptic curves of the D2 and D3 differential equations by $\mathcal{E}_2$ and $\mathcal{E}_3$ respectively.

\begin{center}
\begin{tabular}{ccccc|ccccc}
&&$D2$&&&&&$D3$&&\\
&&&&&&&&&\\
$\a_2$&$\a_1$& $\b$&$j(\mathcal{E}_2)$ & $N(\mathcal{E}_2)$ & $\a_2$&$\a_0$& $\b_0$&$j(\mathcal{E}_3)$ & $N(\mathcal{E}_3)$\\
\hline
&&&&&&&&&\\
-7&-8&2&  $\frac{1556068}{81}$&$24$&  -28& -128& -4&$\frac{207646}{6561}$&$24$\\
&&&&&&&&&\\
-9&27&3&  $0$&$144$&  -36& 432& -12&$54000$&$144$\\
&&&&&&&&&\\
-10&9&3&  $\frac{1556068}{81}$&$24$&  -40& 144& -8&$\frac{35152}{9}$&$24$\\
&&&&&&&&&\\
-11&-1&3& $\frac{488095744}{125}$ &$20$&  -44& -16& -4&$-\frac{20720464}{15625}$&$20$\\
&&&&&&&&&\\
-12&32&4&  $1728$&$32$&  -48& 512& -16&$287496$&$32$\\
&&&&&&&&&\\
-17&72&6&  $\frac{1556068}{81}$&$48$&  -68& 1152& -28&$\frac{3065617154}{9}$&$48$\\
&&&&&&&&&\\
0&-16&0&  $1728$&$32$&  0& -256& 0&$1728$&$32$\\
\end{tabular}
\end{center}

\bigskip
\noindent 
The spectral elliptic curves appear to be non--isomorphic in general but their levels coincide. 
In fact, the respective spectral curves are isogenous, and therefore their $L$-functions are equal. Indeed, the spectral curve of D2
\[
\mathcal{E}_2 \,: \quad y^2 \= z^3 + \a_2 z^2 + \a_1 z  
\]
is isogenous to the elliptic curve
\[
\mathcal{E}_2' \,: \quad y^2 \= z^3 - 2 \a_2 z^2 + (\a_2^2-4 \a_1)z\,,
\]
the isogeny of degree 2  given by
\[\bal
 \mathcal{E}_2 &\rightarrow \mathcal{E}_2' \\
 (z,y) &\mapsto \Bigl( \frac{y^2}{z^2}, \frac{y(\a_1 - z^2)}{z^2} \Bigr)
\eal\]
(see Example 4.5 in \cite{Silverman}). This latter curve is in turn isomorphic to the spectral elliptic curve of the respective D3
\[
\mathcal{E}_3 \,: \quad y^2 \= z^3 + ( -\frac{16}3 \a_2^2 - 64 \a_1) z + (\frac{128}{27}\a_2^3 - \frac{512}3\a_1\a_2)
\]
with the map
\[\bal
 \mathcal{E}_2' &\rightarrow \mathcal{E}_3 \\
 (z,y) &\mapsto \Bigl(\frac14 z+\frac23\a_2, \frac18 y \Bigr)\,.
\eal\]

\end{document}